\documentclass[a4paper]{article}
\usepackage{amsmath,amssymb,amsthm}
\usepackage{hyperref}

\usepackage{tikz}
\usetikzlibrary{matrix, arrows}

\newcommand{\Wh}[1]{\ensuremath{\mathrm{Wh}\left(#1\right)}}
\newcommand{\union}{\cup}
\newcommand{\intersection}{\cap}
\newcommand{\boundary}{\partial_\infty}
\newcommand{\cH}{\check{\mathrm{H}}}
\newcommand{\cC}{\check{\mathrm{C}}}
\newcommand{\hull}{\mathrm{hull}}

\newtheorem{lem}{Lemma}
\newtheorem{cor}{Corollary}
\newtheorem{thm}{Theorem}
\newtheorem{rem}{Remark}

\begin{document}
\title{Computing the \v{C}ech cohomology of decomposition spaces}
\author{Benjamin Barrett}
\maketitle
\bibliographystyle{plain}

\begin{abstract} A line pattern in a free group $F$ is defined by a malnormal
collection of cyclic subgroups. Otal defined a decomposition space
$\mathcal{D}$ associated to a line pattern. We provide an algorithm that
computes a presentation for the \v{C}ech cohomology of $\mathcal{D}$, thought
of as a $F$-module. This answers a relative version of a question of Epstein
about boundaries of hyperbolic groups. \end{abstract}

\section{Introduction}

Epstein asked whether there is an algorithm that computes the Cech cohomology
of the Gromov boundary of a hyperbolic group $\Gamma$, thought of as a
$\Gamma$-module.  Our purpose is to answer a relative version of this question
in the case of Otal's decomposition space, a special case of the Bowditch
boundary of a relatively hyperbolic group.

Fix a free group $F$ of rank $n$ with free basis $\mathcal{B}$. Let
$\mathcal{T}$ be the corresponding Cayley graph containing a vertex $e$
corresponding to the identity in $F$. Fix a finite set $\left\{w_i
\right\}$ of words in $F$. Then the line pattern $\mathcal{L}$ associated
to this set is defined to be the set of lines $\left\{gw_i^k\right\}_k$, $g
\in F$. Let $\mathcal{D}$ be the associated decomposition space: this is the
quotient of $\boundary F$ by the equivalence relation that identifies the two
end points of each line in $\mathcal{L}$. These objects are defined in
\cite{Otal}. Let $q \colon \boundary F \to \mathcal{D}$ be the quotient
projection. Equivalently, $\mathcal{D}$ is the Bowditch boundary of the
relatively hyperbolic group $(F, \mathcal{P})$ where $\mathcal{P}$ is a
peripheral family of cyclic subgroups \cite{Manning}. We present an algorithm
that computes the \v{C}ech cohomology of $\mathcal{D}$.

$F$ acts on $\boundary F$ and this action descends to an action on
$\mathcal{D}$. This gives the cohomology groups of $\mathcal{D}$ the
structure of a right $F$-module. As a corollary to our main result, we shall
see that the \v{C}ech cohomology of $\mathcal{D}$ is finitely presented as an
$F$-module. 

The \v{C}ech cohomology of $\mathcal{D}$ is defined to be the direct limit
over open covers $\mathcal{U}$ that provide successively better combinatorial
approximations to $\mathcal{D}$ of the singular cohomology of the nerve of
$\mathcal{U}$. In Section~\ref{sec:whitehead} we shall see how to associate
an open cover of $\mathcal{D}$ to a finite subtree of $\mathcal{T}$. Then
refining to a finer open cover of $\mathcal{D}$ corresponds to taking a
larger subtree of $\mathcal{T}$. We shall see that the combinatorial
properties of this open cover can be read from the Whitehead graphs associated
to the subtrees and that any open cover of $\mathcal{D}$ can be refined to an
open cover of this form. These open covers have no triple intersections, so we
immediately see that the \v{C}ech cohomology is concentrated in the 0th and 1st
dimensions.

In Sections~\ref{sec:0th} and~\ref{sec:1st} we show how to compute the 0th and
1st \v{C}ech cohomology respectively. Our methods rely on showing that some
(large) finite subtree of $\mathcal{T}$ contains sufficient information to
compute the \v{C}ech cohomology. This approach is based on the proof of
\cite[Lemma 4.12]{CM}. As corollaries we show that there are algorithms that
detect the connectedness of $\mathcal{D}$ and the triviality of
$\cH{}^1\left(\mathcal{D}, \mathbb{Z} \right)$.  The former of these
corollaries is proved by a different argument in \cite{CM}.

\section{Whitehead graphs and open covers}\label{sec:whitehead}

For $\mathcal{X}$ a subtree of $\mathcal{T}$, let \Wh{\mathcal{X}} be the
Whitehead graph of $\mathcal{L}$ at $\mathcal{X}$ as defined in \cite{CM};
briefly, it is a graph with a vertex corresponding to each vertex of
$\mathcal{T}$ adjacent to $\mathcal{X}$ and an edge connecting a pair of
vertices for each line in $\mathcal{L}$ between that pair. For more
information about Whitehead graphs and their applications, see~\cite{Manning2}.

For $v \in T$ let $S^e\left(v\right) \subset \boundary \mathcal{T}$ be the
shadow of $v$ from $e$ as defined in~\cite{CM}: the set of boundary points
$\xi$ such that the geodesic $\left[e, \xi\right]$ contains $v$. These
sets are open and closed and the collection of such sets as $v$ varies in
$\mathcal{T}$ is a basis for the topology on $\boundary\mathcal{T}$.

\begin{lem}\label{lem:construction} Let $\mathcal{X}$ be a finite subtree of
$\mathcal{T}$ containing $e$.  Then there is a covering of $\mathcal{D}$
by a collection of open sets $U_i$ in bijection with the vertices $a_i$ of
$\Wh{\mathcal{X}}$ such that:
\begin{itemize}
\item $q\left(S^e\left(a_i\right)\right) \subset U_i$, 
\item $U_i \intersection U_j \neq \emptyset$ iff there is an edge connecting
  $a_i$ and $a_j$ in $\Wh{\mathcal{X}}$, and
\item there are no triple intersections.
\end{itemize}
\end{lem}

\begin{proof}
We aim to construct open sets $V_i$ covering $\boundary\mathcal{T}$ such
that
\begin{itemize}
\item $V_i$ contains $S^e\left(a_i\right)$,
\item $V_i \intersection V_j = \emptyset$ iff there is an edge connecting
  $a_i$ and $a_j$ in $\Wh{\mathcal{X}}$, 
\item there are no triple intersections, and
\item for each line $l$ in the line pattern, each $V_i$ contains either
  both of $l^{\pm\infty}$ or neither.
\end{itemize}
Then the projection of these sets in $\mathcal{D}$ satisfies the requirements
of the lemma.

We build these inductively. For the first step, take $V_i^0 =
S^e\left(a_i\right)$. Then, for each $i$, there are finitely many lines in
the line pattern passing through $a_i$. Add to $V_i^0$ an open
neighbourhood of the end point not in $V_i^0$ of each such line to obtain
$V_i^1$. We do this in a way as to ensure that $V_i^1$ is the union of
$V_i^0$ and finitely many other shadows, that the open sets added are all
disjoint, and that no line in the line pattern has an end in two different
added sets. This is possible since if a subset of $\boundary \mathcal{T}$ is
a union of finitely many shadows then only finitely many lines in the line
pattern have exactly one end in that subset.

After each $V_i^1$ is defined, continue inductively, ensuring that each
$V_i^k$ is the union of finitely many shadows, so that if a line in the line
pattern has one end in $V_i^k$ then its other end is in $V_i^{k+1}$. We can
do this without introducing any new intersections, so all intersections
correspond to lines from $S^e\left(a_i\right)$ to $S^e\left(a_j\right)$
for some $i$ and $j$, so all intersections correspond to edges in the
Whitehead graph and there are no triple intersections.

Then let $U_i = q\left(\union_k V_i^k\right)$; these sets cover $D$ and
have the required properties. \end{proof}

For $\mathcal{X}$ a finite subtree of $\mathcal{T}$, we shall denote by
$\mathcal{U}_\mathcal{X}$ a cover of $\mathcal{D}$ associated to
$\mathcal{X}$. Then if $\mathcal{X} \subset \mathcal{X}'$,
$\mathcal{U}_{\mathcal{X}'}$ can be chosen to be a refinement of
$\mathcal{U}_\mathcal{X}$. Note that refinement between different open covers
associated to $\mathcal{X}$ as in Lemma~\ref{lem:construction} induces a
natural isomorphism between the singular cochain complexes of the nerves of
those covers.

It will be convenient to define an open cover associated to the empty subtree
of $\mathcal{X}$: this is the trivial covering $\mathcal{U}_\emptyset =
\left\{\mathcal{D}\right\}$. 

\begin{lem}\label{lem:refining} Let $\mathcal{W}$ be a finite open cover of
$\mathcal{D}$.  Then some refinement of $\mathcal{W}$ is of the form given
in Lemma~\ref{lem:construction}.
\end{lem}
\begin{proof}
Let $\mathcal{V}$ be the pullback of $\mathcal{W}$ to
$\boundary\mathcal{T}$. Consider the set
\begin{equation}
C = \left\{a \in \mathcal{T} \vert S^e\left(a\right)
  \subset V \text{ for some } V \in \mathcal{V}\right\}\textup{.}
\end{equation}

The collection $\left\{S^e\left(x\right) \vert x \in \mathcal{T}\right\}$ is
a basis for the topology on $\boundary\mathcal{T}$ so sets of the form
$S^e\left(a\right)$, $a\in C$, cover each $V \in \mathcal{V}$. Hence such
sets cover $\boundary\mathcal{T}$.

$\boundary\mathcal{T}$ is compact, so there is a finite set of points $a_1,
\dots , a_n$ such that $\left\{S^e\left(a_i\right) \right\}$ covers
$\boundary\mathcal{T}$ and each $S^e \left(a_i\right)$ is contained in some
$V_{\sigma\left(i\right)} \in \mathcal{V}$. Let $\mathcal{H}$ be the convex
hull of $\left\{a_i\right\} \union \left\{e\right\}$. Call vertices of
$\mathcal{H}$ adjacent to vertices in $\mathcal{T} - \mathcal{H}$ boundary
vertices. If we take $\left\{a_i\right\}$ to be minimal with its covering
property then the set of boundary points of $\mathcal{H}$ is precisely
$\left\{a_i\right\}$.  Let $\mathcal{X}$ be the subtree of $\mathcal{H}$
obtained by pruning off its boundary vertices.

Let $\mathcal{U}_\mathcal{X} = \left\{U_i\right\}$ be the finite cover of
$\mathcal{D}$ corresponding to $\mathcal{X}$ as in
Lemma~\ref{lem:construction}.  Define a new set $\mathcal{U}'$ of open
subsets of $\mathcal{D}$ by 
\begin{equation}
\mathcal{U}' = \left\{U_i \cap q\left(V_{\sigma\left(i\right)}\right)\right\}
  \text{.}
\end{equation}

$\mathcal{U'}$ covers $\mathcal{D}$ since it covers
$q\left(S^e\left(a_i\right)\right)$ for each $i$. It
is certainly a refinement of $\mathcal{W}$ and it is easy to check that it
corresponds to $\mathcal{X}$ in the sense of the statement of
Lemma~\ref{lem:construction}.
\end{proof}

The results of this section together imply the following corollary:

\begin{cor}
\begin{equation}
\cH{}^n\left(\mathcal{D}, \mathbb{Z}\right) = \varinjlim_\mathcal{X}\cH{}^n
  \left(\mathcal{U}_\mathcal{X}, \mathbb{Z}\right)
\end{equation}
with subtrees $\mathcal{X}$ ordered by inclusion. \qed \end{cor} 

Hence the \v{C}ech cohomology of the decomposition space is determined by the
finite Whitehead graphs.

\section{Computing \texorpdfstring{$\cH{}^0\left(\mathcal{D}, \mathbb{Z}
\right)$}{H0(D,Z)}} \label{sec:0th}

For each element $\left[\sigma\right] \in \cH{}^0\left(\mathcal{D},
\mathbb{Z}\right)$ there exists a subtree $\mathcal{X} \subset \mathcal{T}$
such that $\left[\sigma\right]$ is represented by some
\begin{equation}
\sigma \in \cH{}^0\left(\mathcal{U}_\mathcal{X}, \mathbb{Z}\right) = \ker\left(
d^0 \colon \cC{}^0\left(\mathcal{U}_\mathcal{X}, \mathbb{Z}\right) \to
\cC{}^1\left(\mathcal{U}_\mathcal{X}, \mathbb{Z}\right)\right);
\end{equation}
such an element is an assignment of an integer to each connected component of
\Wh{\mathcal{X}}. In this situation we shall say that $\left[\sigma\right]$
is supported on $\mathcal{X}$ and we shall refer to the minimal such subtree
as the support of $\left[\sigma\right]$. A unique minimal such subtree exists
by the following lemma:

\begin{lem}\label{lem:intersections} Suppose that $\left[\sigma\right] \in
\cH{}^0\left(\mathcal{D}, \mathbb{Z}\right)$ is supported on $\mathcal{X}_1$
and on $\mathcal{X}_2$. Then it is also supported on $\mathcal{X}_1
\intersection \mathcal{X}_2$.\end{lem}

\begin{proof} We prove this by induction on the number of vertices in the
symmetric difference of $\mathcal{X}_1$ and $\mathcal{X}_2$. If the
symmetric difference is non-empty then without loss of generality
$\mathcal{X}_1$ has a leaf $v$ that is not contained in $\mathcal{X}_2$.
It is easy to see that $\left[\sigma\right]$ is supported on $\mathcal{X}_1
- v$.\end{proof}

As discussed in the introduction, $F$ acts on $\mathcal{D}$ by
homeomorphisms, giving the \v{C}ech cohomology the structure of a right
$F$-module. In terms of Whitehead graphs, any $g\in F$ induces a map
\begin{equation}
g \colon \cH{}^0\left(\mathcal{U}_\mathcal{X}, \mathbb{Z}\right) \to
\cH{}^0\left(g^{-1}\mathcal{U}_\mathcal{X}, \mathbb{Z}\right) = 
\cH{}^0\left(\mathcal{U}_{g^{-1}\mathcal{X}}, \mathbb{Z}\right);
\end{equation}
this map takes an element represented by a $\mathbb{Z}$-labelling of the
connected components of \Wh{\mathcal{X}} to the element represented by the
translate by $g^{-1}$ of this diagram. 

We now aim to find an algorithm that computes a presentation for this
$F$-module. First we describe an algorithm that computes a generating set.
The argument is loosely based on the proof of lemma {4.12} in \cite{CM}.

\begin{thm}\label{thm:generators0}
There exists a finite number $N$, computable from $n$ and $\mathcal{L}$,
such that $\cH{}^0\left(\mathcal{D}, \mathbb{Z}\right)$ is generated as an
abelian group by 0-cycles supported on subtrees of $\mathcal{T}$ with at
most $N$ vertices.
\end{thm}

\begin{proof} Let $\left[\sigma\right]$ be a 0-cycle supported on a subtree
$\mathcal{X}$ of $\mathcal{T}$ with more than $N$ vertices, where $N$
is to be chosen later. Then by induction it is sufficient to show that
$\left[\sigma\right]$ can be written as the sum of 0-cycles supported on
strictly smaller subtrees.  The idea is that if $N$ is large enough then
there will be two vertices in $\mathcal{X}$ at which $\left[\sigma\right]$
looks similar and cutting out everything between these vertices allows us to
split $\left[\sigma\right]$ into strictly smaller summands.

$\cH{}^0\left(\mathcal{U}_\mathcal{X}, \mathbb{Z}\right)$ is generated by
0-cycles represented by Whitehead diagrams at $\mathcal{X}$ with one connected component
labelled 1 and the others labelled 0; we can assume without loss of generality
that $\left[\sigma\right]$ is such a 0-cycle. Then $\left[\sigma\right]$ is
can be thought of as a partition of $\Wh{\mathcal{X}}$ into a connected component
and its complement. An example of such a partition is shown pictorially in
figure~\ref{fig:partition}.

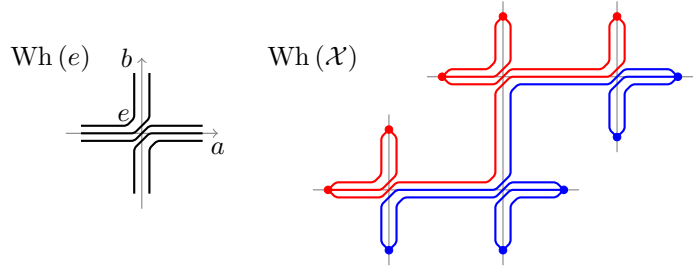
\begin{figure}
\centering

\begin{tikzpicture}[scale=1]

\begin{scope}[shift={(-2.75,0)}] 
  \begin{scope}[gray]
    \draw[->] (-1,0) -- (1,0) node[black,anchor=north] {\(a\)};
    \draw[->] (0,-1) -- (0,1) node[black,anchor=east] {\(b\)};
  \end{scope}

  \draw (-1.2,1) node {\(\Wh{e}\)}; 
  \draw (-0.24,0.24) node {\(e\)}; 

  \begin{scope}[thick,rounded corners=1pt] 
    \draw (-0.8,0.1) -- (-0.2,0.1) -- (-0.1,0.2) -- (-0.1,0.8);
    \draw (-0.8,0) -- (-0.1,0) -- (0.1,0.2) -- (0.1,0.8);
    \draw (-0.8,-0.1) -- (-0.1,-0.1) -- (0.1,0.1) -- (0.8,0.1);
    \draw (-0.1,-0.8) -- (-0.1,-0.2) -- (0.1,0) -- (0.8,0);
    \draw (0.1,-0.8) -- (0.1,-0.2) -- (0.2,-0.1) -- (0.8,-0.1);
  \end{scope}
\end{scope}

\begin{scope}[shift={(2,-0.75)}] 

  \draw (-2.5,1.75) node {\(\Wh{\mathcal{X}}\)}; 

  \begin{scope}[gray] 
    \draw (-2.5,0) -- (1,0);
    \draw (-1,1.5) -- (2.5,1.5);
    \draw (-1.5,1) -- (-1.5,-1);
    \draw (0,2.5) -- (0,-1);
    \draw (1.5,2.5) -- (1.5,0.5);;
  \end{scope}

  \begin{scope} 
    [disc/.style={circle,draw=red,fill=red,very thin,inner sep=0pt,minimum
      size=1mm}]
    \node at (-0.8,1.5) [disc] {};
    \node at (0,2.3) [disc] {};
    \node at (1.5,2.3) [disc] {};
    \node at (-2.3,0) [disc] {};
    \node at (-1.5,0.8) [disc] {};
  \end{scope}

  \begin{scope} 
    [disc/.style={circle,draw=blue,fill=blue,very thin,inner sep=0pt,minimum
      size=1mm}]
    \node at (-1.5,-0.8) [disc] {};
    \node at (0,-0.8) [disc] {};
    \node at (0.8,0) [disc] {};
    \node at (2.3,1.5) [disc] {};
    \node at (1.5,0.7) [disc] {};
  \end{scope}

  \begin{scope}[thick,red,rounded corners=1pt] 
    \draw (-2.3,0) -- (-2.2,0.1) -- (-1.7,0.1) -- (-1.6,0.2) -- (-1.6,0.7)
      -- (-1.5,0.8);
    \draw (-2.3,0) -- (-1.6,0) -- (-1.4,0.2) -- (-1.4,0.7) -- (-1.5,0.8);
    \draw (-0.8,1.5) -- (-0.7,1.6) -- (-0.2,1.6) -- (-0.1,1.7) -- (-0.1,2.2)
      -- (0,2.3);
    \draw (-0.8,1.5) -- (-0.1,1.5) -- (0.1,1.7) -- (0.1,2.2) -- (0,2.3);
    \draw (-2.3,0) -- (-2.2,-0.1) -- (-1.6,-0.1) -- (-1.4,0.1) -- (-0.2,0.1)
      -- (-0.1,0.2) -- (-0.1,1.3) -- (0.1,1.5) -- (1.4,1.5) -- (1.6,1.7)
      -- (1.6,2.2) -- (1.5,2.3);
    \draw (-0.8,1.5) -- (-0.7,1.4) -- (-0.1,1.4) -- (0.1,1.6) -- (1.3,1.6)
      -- (1.4,1.7) -- (1.4,2.2) -- (1.5,2.3);
  \end{scope}

  \begin{scope}[thick,blue,rounded corners=1pt] 
    \draw (0,-0.8) -- (-0.1,-0.7) -- (-0.1,-0.2) -- (0.1,0) -- (0.8,0);
    \draw (0,-0.8) -- (0.1,-0.7) -- (0.1,-0.2) -- (0.2,-0.1) -- (0.7,-0.1)
      -- (0.8,0);
    \draw (1.5,0.7) -- (1.4,0.8) -- (1.4,1.3) -- (1.6,1.5) -- (2.3,1.5);
    \draw (1.5,0.7) -- (1.6,0.8) -- (1.6,1.3) -- (1.7,1.4) -- (2.2,1.4)
      -- (2.3,1.5);
    \draw (-1.5,-0.8) -- (-1.4,-0.7) -- (-1.4,-0.2) -- (-1.3,-0.1) -- (-0.1,-0.1)
      -- (0.1,0.1) -- (0.7,0.1) -- (0.8,0);
    \draw (-1.5,-0.8) -- (-1.6,-0.7) -- (-1.6,-0.2) -- (-1.4,0) -- (-0.1,0)
      -- (0.1,0.2) -- (0.1,1.3) -- (0.2,1.4) -- (1.4,1.4) -- (1.6,1.6)
      -- (2.2,1.6) -- (2.3,1.5);
  \end{scope}
\end{scope}

\end{tikzpicture}

\caption{A partition of a disconnected Whitehead graph into two connected
components. An assignment of the integer 1 to the red part and 0 to the blue
part represents an element of $\cH{}^0\left(\mathcal{D}, \mathbb{Z}\right)$.
In this example $F$ is free on two generators and the line pattern is
generated by the word $a^2bab$.}
\label{fig:partition}
\end{figure}

Suppose that $N$ is large enough that any subtree of $\mathcal{T}$ with
more than $N$ vertices is guaranteed to contain an embedded arc of length at
least $M+2$, where $M$ is a computable function of $n$ and
$\mathcal{L}$ to be chosen later. Then let $v_1, \ldots, v_M$ be the
interior vertices of such an embedded arc in $\mathcal{X}$. Traversing this
arc in the direction from $v_1$ to $v_M$, record for each vertex $v_i$ an
ordered pair $\left(s_i, t_i\right)$ of elements of $\mathcal{B}^\pm$,
where $s_i$ labels the incoming edge at $v_i$ of the embedded arc at and
$t_i$ labels the outgoing edge.

Suppose that $M$ is large enough that at least $K$ of these pairs are
equal. Here $K$ is a computable function of $\mathcal{L}$ to be chosen
later. Then let $v_{i_1}, \ldots, v_{i_K}$ be vertices with equal associated
edge pairs. The edges of each $\Wh{v_{i_j}}$ extend to edges in
$\Wh{\mathcal{L}}$, hence the partition of \Wh{\mathcal{L}} associated to
$\left[\sigma\right]$ gives a partition of the edges of $\Wh{v_{i_j}}$ into
a subset and its complement. 

Treating the $v_{i_j}$ as elements of $F$, the translation of
$\Wh{v_{i_j}}$ by $v_{i_j}^{-1}$ gives a partition on the edges of
$\Wh{e}$. There is a finite number of such partitions; let $K$ be greater
than that number. Then we obtain $v, w=g\left(v\right) \in \left\{v_{i_1},
\ldots v_{i_K}\right\}$ such that these translates of the associated
partitions agree.

Now we define two disjoint subsets of $\mathcal{X}$. Let $A$ be the
vertices $u \neq v$ of $\mathcal{X}$ such that the geodesic in
$\mathcal{T}$ from $w$ to $u$ passes through $v$, and let $B$ be the
same with the r\^oles of $v$ and $w$ reversed. Without loss of generality,
$A$ contains at least as many vertices as $B$ does. Then let $A' = A
\union \left\{v\right\}$. See figure~\ref{fig:cutting}.

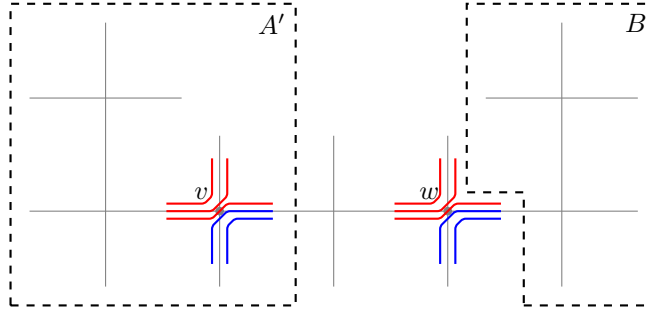
\begin{figure}
\centering
\begin{tikzpicture}[scale=1]

\begin{scope}[gray] 
  \draw (-4,0) -- (4,0);
  \draw (-4,1.5) -- (-2,1.5);
  \draw (2,1.5) -- (4,1.5);
  \draw (-3,-1) -- (-3,2.5);
  \draw (-1.5,-1) -- (-1.5,1);
  \draw (0,-1) -- (0,1);
  \draw (1.5,-1) -- (1.5,1);
  \draw (3,-1) -- (3,2.5);
\end{scope}

\begin{scope} 
  [disc/.style={circle,draw=gray,fill=gray,very thin,inner sep=0pt,minimum
    size=1mm}]
  \node at (-1.74,0.24) {\(v\)};
  \node at (1.26,0.24) {\(w\)};
  \node at (-1.5,0) [disc] {};
  \node at (1.5,0) [disc] {};
\end{scope}

\begin{scope}[thick,rounded corners=1pt] 
  \foreach \x in {-1.5,1.5} {
    \begin{scope}[red] 
      \draw (-0.7+\x,0.1) -- (-0.2+\x,0.1) -- (-0.1+\x,0.2) -- (-0.1+\x,0.7);
      \draw (-0.7+\x,0) -- (-0.1+\x,0) -- (0.1+\x,0.2) -- (0.1+\x,0.7);
      \draw (-0.7+\x,-0.1) -- (-0.1+\x,-0.1) -- (0.1+\x,0.1) -- (0.7+\x,0.1);
    \end{scope}

    \begin{scope}[blue] 
      \draw (0.7+\x,-0.1) -- (0.2+\x,-0.1) -- (0.1+\x,-0.2) -- (0.1+\x,-0.7);
      \draw (0.7+\x,0) -- (0.1+\x,0) -- (-0.1+\x,-0.2) -- (-0.1+\x,-0.7);
    \end{scope}
  }
\end{scope}

\begin{scope}[thick,dashed]
  \draw (-4.25,-1.25) rectangle (-0.5,2.75) node[anchor=north east] {\(A'\)};
  \draw (2.5,-1.25) -- (4.25,-1.25) -- (4.25,2.75) -- (1.75,2.75) --
    (1.75,0.25) -- (2.5,0.25) -- cycle;
  \draw (4.25,2.75) node[anchor=north east] {\(B\)};
\end{scope}

\end{tikzpicture}
\caption{$v$ and $w$ are chosen so that $\sigma$ induces the same
partition on \Wh{v} as on \Wh{w}. Then $\tau$ is an element of
$\cH^0\left(\mathcal{U}_{A \union vw^{-1} B}, \mathbb{Z}\right)$ chosen to
induce the same partition on \Wh{A'} as $\sigma$ does.}
\label{fig:cutting}
\end{figure}

We now cancel off the part of $\left[\sigma\right]$ supported on $A$.  Let
$\mathcal{Y} = A' \union vw^{-1} B$, where we treat $v$ and $w$ as
elements of the group $F$. As described above, $\sigma$ induces a partition
on the edges of $\Wh{u}$ for each vertex $u$ in $A'$ and $B$, and hence
by translation on $\Wh{u}$ for each vertex $u$ in $\mathcal{Y}$. The
partitions at $v \in A'$ and the vertex of $vw^{-1}B$ adjacent to $v$ are
consistent with respect to the splicing map, so we obtain a partition of the
graph $\Wh{\mathcal{Y}}$. Hence we can define $\tau \in
\cH^{0}\left(\mathcal{Y}, \mathbb{Z}\right)$ to be a cycle represented by
assigning the integer 1 to one component of $\Wh{\mathcal{Y}}$ and 0 to the
others in a way that agrees at the vertices of $A'$ with the labelling of
components represented by $\sigma$.

Then $\left[\tau\right]$ is supported on $\mathcal{Y}$, which contains at
most $\left|A'\right| + \left|B\right| < \left|\mathcal{X}\right|$ vertices,
and $\left[\sigma\right]-\left[\tau\right]$ is supported on $\left(
\mathcal{X} - A\right) \union g^{-1}B$, which has fewer vertices than
$\mathcal{X}$ since $A'$ has more vertices than $B$. \end{proof}

\begin{rem} We can give bounds on the function $N\left(\mathcal{L}, n\right)$
explicitly: let $k$ be the number of edges in $\Wh{e}$; this is equal to
the sum of the lengths of the words that generate $\mathcal{L}$. Then:
\begin{align}
N &= \left(2n\right)^{\left(2n\right)^2\left(2^k+1\right)+1}\left(1+2n\right)+1
\end{align}
\end{rem}

\begin{cor} $\cH{}^0\left(\mathcal{D}, \mathbb{Z}\right)$ has a computable
finite generating set.\end{cor}

\begin{proof} Theorem~\ref{thm:generators0} implies that $\cH{}^0\left(
\mathcal{D}, \mathbb{Z} \right)$ is generated as an $F$-module by the set of
0-cycles supported on subtrees of the ball of radius $N$ centred at $e$.
If $\mathcal{X}$ is this ball then $\cH{}^0\left(\mathcal{U}_\mathcal{X},
\mathbb{Z}\right)$ has a computable finite generating set as a
$\mathbb{Z}$-module, since $\Wh{\mathcal{X}}$ can be partitioned into
its connected components algorithmically.\end{proof}

\begin{cor} There is an algorithm that determines whether or not
$\mathcal{D}$ is connected. \end{cor}

This corollary is proved by a different argument in~\cite{CM}. In that paper it
is shown that, after simplifying $\Wh{e}$ as much as possible using Whitehead
moves, $\mathcal{D}$ is connected if and only if $\Wh{e}$ is connected.

\begin{proof} $\mathcal{D}$ is connected if and only if $\cH{}^0\left(
\mathcal{D}, \mathbb{Z}\right)$ is generated by the cochain supported on the
trivial covering that assigns the integer $1$ to the only open set in that
covering; in this case it is isomorphic to $\mathbb{Z}$ with trivial $F$
action. Equivalently, $\mathcal{D}$ is connected if and only if any $\sigma
\in \cH{}^0\left(\mathcal{U}_\mathcal{X}, \mathbb{Z}\right)$ is represented by
the assignment of the same integer to each component of $\Wh{\mathcal{X}}$
for all subtrees $\mathcal{X} \subset \mathcal{T}$.  It is sufficient to
check this on a generating set, and we have already shown that $\cH{}^0\left(
\mathcal{D}, \mathbb{Z}\right)$ has a computable finite generating
set.\end{proof}

We now have an algorithm that gives a finite set $\left[\sigma_1\right],
\ldots , \left[\sigma_k\right]$ of cohomology classes that generate
$\cH^0\left(\mathcal{D}, \mathbb{Z}\right)$ as a right $F$-module. This is
equivalent to a surjection $\left(p \colon \mathbb{Z}F^k \to \cH^0 \left(
\mathcal{D}, \mathbb{Z}\right)\right)$ of right $F$-modules. Let $e_i$ be
the $i$th basis vector in the free module, and let it be mapped to
$\left[\sigma_i\right]$ under $p$. To complete the computation of a
presentation for of $\cH^0\left(\mathcal{D}, \mathbb{Z}\right)$ we need an
algorithm that computes a generating set for the kernel of $p$.

For each $\left[\sigma_i\right]$ let $\mathcal{X}_i$ be the support of
$\left[\sigma_i\right]$ is supported. A general element $x \in
\mathbb{Z}F^k$ is of the form 
\begin{align}\label{eqn:freemodule}
x = \sum\limits_{i, j} n_{ij}\left(e_{j}g_{ij}\right), \text{where $n_{ij} \in
  \mathbb{Z}$, $g_{ij} \in F$}.
\end{align}
Define the support of $x$ to be 
\begin{align}
\hull \left(\bigcup\limits_{i, j} g_{ij}^{-1}\mathcal{X}_j\right).
\end{align}
Note that the support of $px$ is contained in the support of $x$.

We can now state and prove a theorem that shows that the kernel of $p$ is
generated by elements of bounded size, in the same way that
Theorem~\ref{thm:generators0} shows that $\cH{}^0\left(\mathcal{D},
\mathbb{Z}\right)$ is generated by elements of bounded size.

\begin{thm}\label{thm:relators0} $\ker p$ is generated as an abelian group by
elements whose supports have at most $N$ vertices, where $N$ is a
computable function of $\mathcal{L}$ and $n$.\end{thm}

\begin{proof} Our approach here is similar to that in the proof of
Theorem~\ref{thm:generators0}: we show that---for sufficiently large
(computable) $N$---an element of $\ker p$ supported on a set with more
than $N$ vertices can be written as the sum of two elements of $\ker p$
supported on strictly smaller sets.

Let $D$ be the maximum of the diameters of the $\mathcal{X}_i$. Let a ball
of diameter $D$ contain $L$ vertices.

From the proof of Theorem~\ref{thm:generators0} it is clear that in picking a
preimage $x$ under $p$ of an element $\left[\sigma\right] \in
\cH{}^0\left( \mathcal{D}, \mathbb{Z}\right)$ it might well be necessary for
the support of $x$ to be strictly larger than the support of
$\left[\sigma\right]$. We will need to be able to bound the size of the
support of $x$ for $\left[\sigma\right]$ supported on a ball of radius
at most $D$. We deal with this first.

With some care, the proof of Theorem~\ref{thm:generators0} gives an explicit
bound. At each step, the cochain is split into two pieces, each supported on a
set with strictly fewer vertices. Hence, since $\left[\sigma\right]$ is
supported on a set with $L$ vertices, it can certainly be written as a
$\mathbb{Z}$-linear combination of at most $2^L$ elements of our generating
set. So if each generator has at most $M$ vertices, any
$\left[\sigma\right]$ supported on a ball of radius $D$ has a preimage
supported on a set with at most $2^L M$ vertices. By construction, this set
can be taken to be connected. Let $K=2^L M$.

Let $N$ be large enough that any subtree $\mathcal{X}$ of $\mathcal{T}$
with at least $N$ vertices contains a vertex $v$ such that $\mathcal{X} -
v$ is the union of two (disconnected) subgraphs of $\mathcal{X}$ each with
at least $K+L$ vertices. For example, this holds if $\mathcal{X}$ is
guaranteed to contain an embedded arc of length at least $2\left(K+L\right)
+1$. Then suppose that some relator $x \in \ker p$ is supported on a subtree
$\mathcal{X} \subset \mathcal{T}$ with at least $N$ vertices. Let $v$ be
as in the definition of $N$. Then we aim to divide $x$ as the sum of two
smaller relators by cutting at $v$.

$x$ is of the form of equation~\ref{eqn:freemodule} and is such that
$g_{ij}^{-1}\mathcal{X}_j \subset \mathcal{X}$ for each pair $i, j$. Let
$A$ and $B$ be the two components of $\mathcal{X} - v$ as described
above, and let $C$ be the ball in $\mathcal{X}$ of radius $D$ centred at
$v$. Let $y \in \mathbb{Z}F^k$ be the sum of those summands of $x$ in
equation~\ref{eqn:freemodule} whose supports are contained in $A$.  Then the
support of $y$ is a subset of $A$ and the support of $x-y$ is a subset of
$B \union C$. 

Roughly, $y$ and $x-y$ will be the two desired smaller relators whose sum
is $x$. However $py \neq 0$, so we shall need to add a small correction
term.  In order to ensure that the correction term is indeed small (in the
sense of having small support) we use Lemma~\ref{lem:intersections}. 

Since $py = -p\left(x-y\right)$, $py$ is supported on $A \intersection
\left(B \union C\right) = A \intersection C$. This is a subtree of a tree of
diameter $2D$, so by assumption $py$ has a preimage $w$ under $p$ that
is supported on a set with at most $K$ vertices. Then $p\left(y-w\right) =
0$ and $x = \left(y-w\right) + \left(x-y+w\right)$ so it remains to show
that $y-w$ and $x-y+w$ have strictly smaller supports than $x$. But the
support of $x$ has $\left|A\right| + \left|B\right| + 1$ vertices, while
$y-w$ and $x-y+w$ are supported on sets with at most $\left|A\right| + K$
and $\left| B \right| + \left|C\right| + K$ vertices respectively.  $\left|
A \right|$ and $\left|B\right|$ have at least $K + \left|C\right|$
vertices, so this completes the proof.\end{proof}

\begin{cor} There is an algorithm that computes a finite presentation for
$\cH{}^0\left(\mathcal{D}, \mathbb{Z}\right)$.\end{cor}

\begin{proof} Note that $F$ acts on $\mathbb{Z}F^k$ by translation in the
sense that if the support of $x \in \mathbb{Z}F^k$ is $\mathcal{X}$ then
the support of $xg$ is $g^{-1}\mathcal{X}$. Hence if $N$ is as in the
statement of Theorem~\ref{thm:relators0} then that theorem shows that $\ker
p$ is generated as an $F$-module by those of its elements that are supported
on a ball of radius $N$ ball centred at $e$. 

In other words, $\ker p$ is generated by its intersection with the set of
those $\mathbb{Z}$-linear combinations of translates of the $\left\{
e_i\right\}$ by $F$ whose supports are contained in this ball of radius
$N$.  To find all such linear combinations is simply to solve a finite
dimensional $\mathbb{Z}$-linear equation, which can be done algorithmically,
for example using Smith normal form.\end{proof}

\section{Computing \texorpdfstring{$\cH{}^1\left(\mathcal{D}, \mathbb{Z}
\right)$}{H1(D,Z)}} \label{sec:1st}

$\cH{}^1\left(\mathcal{U}_\mathcal{X}, \mathbb{Z}\right)$ is the quotient of
$\cC{}^1\left(\mathcal{U}_\mathcal{X}, \mathbb{Z}\right)$ by
$d\cC{}^0\left(\mathcal{U}_\mathcal{X}, \mathbb{Z}\right)$, since
$\cC{}^2\left(\mathcal{U}_\mathcal{X}, \mathbb{Z}\right)$ is trivial. Since
taking direct limits of families of $\mathbb{Z}$-modules is an exact functor,
$\cH{}^1\left(\mathcal{D}, \mathbb{Z}\right)$ is also a quotient:
\begin{center}\begin{tikzpicture}[>=angle 90]
\matrix (a) [matrix of nodes, row sep=3em, column sep = 2em, text height = 
  1.5ex, text depth=0.25ex]
{
  $0$ & 
  $d\varinjlim_\mathcal{X}\cC{}^0\left(\mathcal{U}_\mathcal{X},
    \mathbb{Z} \right)$ & 
  $\varinjlim_\mathcal{X} \cC^1\left( \mathcal{U}_\mathcal{X}, \mathbb{Z}
    \right)$ & 
  $\cH{}^1\left(\mathcal{D}, \mathbb{Z}\right)$ &
  $0$ \\
};
\path[->](a-1-1) edge (a-1-2)
  (a-1-2) edge (a-1-3)
  (a-1-3) edge (a-1-4)
  (a-1-4) edge (a-1-5);
\end{tikzpicture}\end{center}
is exact. As in the previous section, each of these abelian groups can be
endowed with the structure of an $F$-module so that the homomorphisms in the
short exact sequence are homomorphisms of $F$-modules.

Now finding a presentation for $\cH{}^1\left(\mathcal{D}, \mathbb{Z}\right)$
is equivalent to finding a presentation for $\varinjlim_\mathcal{X}
\cC^1\left( \mathcal{U}_\mathcal{X}, \mathbb{Z}\right)$ and a generating set
for $d\varinjlim_\mathcal{X}\cC{}^0\left(\mathcal{U}_\mathcal{X}, \mathbb{Z}
\right)$. We present an algorithm that does the former in
theorems~\ref{thm:generators1} and~\ref{thm:relators1} and an algorithm that
does the latter in Lemma~\ref{lem:image}.

As in the previous section, cochains have a convenient representation in terms
of the Whitehead graph. A 1-cochain (with respect to an open cover
$\mathcal{U}$) is a map that associates an integer to each pair $U_1, U_2
\in \mathcal{U}$ with $U_1 \intersection U_2 \neq \emptyset$.  Equivalently,
if $\mathcal{U}$ is the open cover associated to a Whitehead graph
$\Wh{\mathcal{X}}$, this is the assignment of an integer to each edge in the
Whitehead graph, with the restriction that if two edges connect the same pair
of vertices then they are assigned the same integer. Refinement to the open
cover associated to a larger Whitehead graph preserves the labelling of the old
edges, and assigns the integer 0 to each new edge.

\begin{thm}\label{thm:generators1} There is a computable function $N$ of
$\mathcal{L}$ and $n$ so that $\varinjlim \cC{}^1\left(
\mathcal{U}_\mathcal{X}, \mathbb{Z} \right)$ is generated as an abelian group
by elements supported on sets with fewer than $N$ vertices. \end{thm}

\begin{proof} A subset $\mathcal{X} \subset \mathcal{C}$ gives a partition
$P_\mathcal{X}$ on the edges of $\Wh{v}$ for each vertex $v$ of
$\mathcal{X}$: in this partition, two edges are related if those edges extend
to edges between the same pair of vertices in $\Wh{\mathcal{X}}$. For each
element $a \in \mathcal{B}^{\pm}$ and partition $P$ on the edges of
$\Wh{e}$ there exists a subset $\mathcal{X} \subset \hull\left(e \union
S^e\left(a\right)\right) \intersection \mathcal{T}$ such that $P$ is at
least as fine as $P_\mathcal{X}$. Let $\mathcal{X}_{\left(a, P\right)}$ be
a minimal such subset; it is easy to see that it is contained in any other
subset with this property. 

Let $N$ be the maximum number of vertices in any $\mathcal{X}_{\left(a,
P\right)}$. We now prove that $\varinjlim \cC{}^1\left(
\mathcal{U}_\mathcal{X}, \mathbb{Z}\right)$ is generated as an abelian group
by elements supported on sets with at most $N$ vertices.

Let $\left[\sigma\right]$ be a 1-cochain supported on $\mathcal{X}$ and let
$\sigma \in \cC^1\left(\mathcal{U}_\mathcal{X}, \mathbb{Z}\right) $ represent
$\left[\sigma\right]$. Let $v$ be a leaf of $\mathcal{X}$. Then
$\sigma$ defines a partition $P$ on the edges of $\Wh{v}$ by relating two
edges if they are assigned the same integer by $\sigma$.  $P_\mathcal{X}$
is at least as fine as $P$. By translating $P$ by $v^{-1}$ (considering
the vertex as an element of the group $F$) we obtain a partition on
$\Wh{e}$, which we shall also denote by $P$.  Let $a \in \mathcal{B}$ be
the label on the edge connecting $v$ to the rest of $\mathcal{X}$.

By the definition of $N$, $\mathcal{X}_{\left(a, P\right)}$ has at most
$N$ vertices. Let $\tau$ be a 1-cochain supported on this set that assigns
to each edge of $\Wh{e}$ the same integer that $\sigma v$ does; note that
$\tau$ satisfies the requirement that if two edges connect the same pair of
vertices in the Whitehead graph then they are assigned the same integer. Since
$v\mathcal{X}_{\left(a, P\right)} \subset \mathcal{X}$, $\tau v^{-1}$ is
supported on $\mathcal{X}$ and then it is easy to see that $\sigma - \tau
v^{-1}$ is supported on $\mathcal{X} - v$. Proceeding by induction on the
number of vertices in the support of $\sigma$ we obtain the required
result.\end{proof}

This immediately implies the following corollary:

\begin{cor} $\varinjlim \cC{}^1\left(\mathcal{U}_\mathcal{X}, \mathbb{Z}
\right)$ is generated as an $F$-module by those of its elements that are
supported on a ball centred at $e$ of computable finite diameter.\qed
\end{cor}

To proceed to compute a set of relators for $\varinjlim \cC{}^1\left(
\mathcal{U}_\mathcal{X}, \mathbb{Z}\right)$ we require the following lemma,
which is analagous to Lemma~\ref{lem:intersections}.

\begin{lem}\label{lem:intersections2} Let $D$ be the maximum of the diameters
of the supports of the generators computed in Theorem~\ref{thm:generators1}.
Suppose that $\left[\sigma\right] \in \varinjlim \cC{}^1\left(
\mathcal{U}_\mathcal{X}, \mathbb{Z}\right)$ is supported on $\mathcal{X}_1$
and on $\mathcal{X}_2$ where $\mathcal{X}_1$ and $\mathcal{X}_2$ are
subtrees of $\mathcal{C}$ with non-trivial intersection. Then
$\left[\sigma\right]$ it is also supported on a $D$-neighbourhood of
$\mathcal{X}_1 \intersection \mathcal{X}_2$.\end{lem}

\begin{proof} $\left[\sigma\right]$ is represented by a labelling of the
edges of $\Wh{\mathcal{X}_1 \union \mathcal{X}_2}$ by integers such that each
edge that does not pass through $\mathcal{X}_1 \intersection \mathcal{X}_2$
is labelled by 0. Any such 1-cycle is supported on a $D$-neighbourhood of
this subset.\end{proof}

\begin{thm}\label{thm:relators1} There is an algorithm that computes a set of
relators for the $F$-module $\varinjlim \cC{}^1\left(
\mathcal{U}_\mathcal{X}, \mathbb{Z}\right)$ with respect to the basis computed
by the algorithm of Theorem~\ref{thm:generators1}.\end{thm}

\begin{proof} The proof of Theorem~\ref{thm:relators0} works here too;
Lemma~\ref{lem:intersections2} has a weaker hypothesis than
Lemma~\ref{lem:intersections} does but this makes no difference to the proof.
\end{proof}

Theorems~\ref{thm:generators1} and~\ref{thm:relators1} together give an
algorithm that computes a finite presentation for $\varinjlim\cC{}^1
\left(\mathcal{U}_\mathcal{X}, \mathbb{Z}\right)$.  $\cH{}^1\left(
\mathcal{D}, \mathbb{Z}\right)$ is the quotient of this abelian group by the
image under the boundary map of $\varinjlim\cC{}^0\left(
\mathcal{U}_\mathcal{X}, \mathbb{Z}\right)$, so it remains to show that this
image has a computable generating set.

\begin{lem}\label{lem:image} $d\varinjlim\cC{}^0\left(\mathcal{U}_\mathcal{X},
\mathbb{Z} \right)$ is generated as an $F$-module by those of its elements
that are supported on $\Wh{e}$.\end{lem}

\begin{proof} $\cC{}^0\left(\mathcal{U}_\mathcal{X}, \mathbb{Z} \right)$ is
generated as an abelian group by those elements that are supported on
$\Wh{v}$ for some $v \in \mathcal{X}$, and if $\left[\sigma\right]$ is
supported on $\mathcal{X}$ then so is $d\left[\sigma\right]$.\end{proof}

Putting the results of this section together we conclude:

\begin{thm} There is an algorithm that determines a presentation for the
$F$-module $\cH{}^{1}\left(\mathcal{D}, \mathbb{Z}\right)$.\qed\end{thm}

\begin{cor} There is an algorithm that determines whether or not
$\cH{}^{1}\left(\mathcal{D}, \mathbb{Z}\right)$ is trivial.\end{cor}

\begin{proof} Since $\cH{}^{1}\left(\mathcal{D}, \mathbb{Z}\right)$ has a
computable finite generating set, it is sufficient to be able to determine
whether or not each generator is trivial. But each
$\cH{}^{1}\left(\mathcal{U}_\mathcal{X}, \mathbb{Z}\right)$ includes
injectively into $\cH{}^{1}\left(\mathcal{D}, \mathbb{Z}\right)$, so it is
sufficient to be able to determine whether or not a given element of some
$\cH{}^1\left(\mathcal{U}_\mathcal{X}, \mathbb{Z}\right)$ is trivial; that
is, whether or not it is in the image of $d$. But the problem of determining
whether or not such an element has a preimage under $d$ is equivalent to
determining whether or not some finite dimensional $\mathbb{Z}$-linear
equation has a solution, so can be done algorithmically.  \end{proof}

\bibliography{decompositionspace}

\end{document}